\documentclass[a4paper,10pt]{amsart}

\usepackage[utf8]{inputenc}
\usepackage[top=4.5cm,bottom=4cm,right=3.9cm,left=3.9cm]{geometry}
\usepackage{amsmath,amssymb,amsthm}

\usepackage[pagebackref = true]{hyperref}
\hypersetup{
colorlinks = true,
pdfstartview = FitH,
pdfpagemode = UseNone,
pdfauthor    = {Giacomo Cherubini, Pavlo Yatsyna},
pdftitle     = {},
pdfkeywords  = {},
pdfcreator   = {Pdflatex},
linktoc = none
}

\numberwithin{equation}{section}
\def\eps{\epsilon}

\def\ZZ{\mathbb{Z}}
\def\RR{\mathbb{R}}
\def\res{\mathrm{res}}

\def\ntilde{\tilde{n}}

\theoremstyle{plain}
\newtheorem{theorem}{Theorem}[section]
\newtheorem{proposition}[theorem]{Proposition}
\newtheorem{lemma}[theorem]{Lemma}
\newtheorem{corollary}[theorem]{Corollary}

\title{Potential energy of totally positive algebraic integers}

\author{Giacomo Cherubini}
\author{Pavlo Yatsyna}

\address{
         Charles University,
         Faculty of Mathematics and Physics,
         Department of Algebra,
         Sokolov\-sk\'a 83, 18600 Praha~8,
         Czech Republic
        }

\email{
    cherubini@karlin.mff.cuni.cz\\
    yatsyna@karlin.mff.cuni.cz
    }

\date{\today}

\begin{document}

\begin{abstract}
Given positive real numbers, we prove two inequalities
involving their potential energy and their power sums.
We also prove an inequality involving the energy
and the discriminant and apply it to deduce a result on
totally positive irreducible polynomials.
\end{abstract}

\maketitle

\section{Introduction}

Given $n$ positive real numbers $x_1,\dots,x_n$, there are
several ways to measure their interaction, such as
Maclaurin's inequalities, Newton's inequalities,
or even the simpler inequality between arithmetic mean and
geometric mean:
\begin{equation}\label{eq:AG}
\left(\frac{x_1+\cdots+x_n}{n}\right)^n\geq x_1\cdots x_n.
\end{equation}
When $x_1,\dots,x_n$ are the roots of a monic irreducible polynomial in $\ZZ[x]$,
many of these inequalities can be effectively improved. For instance,
Siegel \cite[Theorems I,II]{siegel} and Hunter \cite[Theorem 1]{hunter}
showed that the right-hand side in \eqref{eq:AG}
can be multiplied by and explicit function strictly greater than one.

Siegel's theorem is based on the study of the discriminant,
an idea already used in a paper by Schur \cite{schur}.
Their results have applications to what is nowadays called
the Schur--Siegel--Smyth trace problem
(see e.g.~\cite{aguirre-peral} for an overview of the problem and \cite{smith,www}
for the current state of the art).

Taking inspiration from Hunter's work, we study instead the \emph{potential energy}
of $x_1,\dots,x_n$, defined as
\begin{equation}
E = \sum_{1\leq i< j \leq n} (x_i-x_j)^2.
\end{equation}
Hunter's theorem states that we can multiply the right-hand side in \eqref{eq:AG}
by a certain explicit function of the potential energy. His proof requires
finding the maximum of~$E$ over the subset of $\RR^n_{+}$ where
the trace $ns=x_1+\cdots+x_n$ and the norm $p=x_1\cdots x_n$ are fixed (see \cite[Lemma 1]{hunter}); then using this maximum, he deduces a refinement of~\eqref{eq:AG}.
Our first observation is that Hunter's proof gives the minimum of $E$ as well.
\begin{proposition}\label{P1}
Let $s,p>0$ with $s^n>p$, and let $x_1,\dots,x_n$ be positive real numbers with $x_1+\cdots+x_n=ns$
and $x_1\cdots x_n=p$. Then
\[
E \geq (n-1)(ns)^2\alpha^2,
\]
where $\alpha$ is the solution in $(-\frac{1}{n-1},0)$ of the equation
\(
(1+\alpha(n-1))(1-\alpha)^{n-1} = s^{-n}p
\).
\end{proposition}
In Section~\ref{S2} we sketch the proof of Proposition~\ref{P1},
which follows closely Hunter's proof of \cite[Lemma 1]{hunter}.
Curiously, when the energy is given and is not too large, we can use Proposition \ref{P1}
to obtain an inequality that goes in the opposite direction of \eqref{eq:AG}.
\begin{corollary}\label{C1}
Let the notation be as in Proposition \ref{P1}.
Assume that $E<\frac{(ns)^2}{n-1}$ and let
$\beta$ be the negative root of $E=(n-1)(ns)^2\beta^2$. Then
\begin{equation}\label{eq:corollary}
\frac{s^n}{p} \leq \frac{1}{(1+\beta(n-1))(1-\beta)^{n-1}}.
\end{equation}
\end{corollary}

Consider now the polynomial
\begin{equation}\label{eq:003}
f(x) = \prod_{k=1}^{n}(x-x_k) = x^n - a_{n-1}x^{n-1} + \cdots + (-1)^{n-1}a_1x + (-1)^na_0.
\end{equation}
Fixing the trace and the norm, as done by Hunter (and Siegel),
amounts to fixing the coefficients $a_{n-1}$ and $a_0$.

The main goal of the present paper is to show that one can
use different functions of the roots
and derive a result similar to Hunter's theorem.
We will work with the power sums
\[
S_r = x_1^r + \cdots + x_n^r, \qquad r\geq 1.
\]
Let us point out that $S_1$ is the
trace, while knowing $S_1,\dots,S_r$ allows one to recover
the coefficients $a_{n-1},\dots,a_{n-r}$
in \eqref{eq:003} by means of Newton's identities.

\begin{theorem}\label{T1}
Let $r\in\mathbb{N}$, $r\geq 3$, and let $S_1,S_r$ be positive real numbers
satisfying $S_r\leq S_1^r\leq n^{r-1}S_r$.
Let $x_1,\dots,x_n\geq 0$ with $x_1+\cdots+x_n=S_1$, $x_1^r+\cdots+x_n^r=S_r$
and with potential energy~$E$. Then,
\begin{equation}\label{T1:eq001}
E \geq E_{min} = (n-1)S_1^2 \alpha^2,
\end{equation}
where $\alpha$ is the root in $[0,1)$ of the equation
$(1+\alpha(n-1))^r+(n-1)(1-\alpha)^r=n^rS_rS_1^{-r}$.
Conversely, if the energy $E$ is given, then we have
\begin{equation}\label{T1:eq003}
S_r \leq ((1+\beta(n-1))^r + (n-1)(1-\beta)^r) \frac{S_1^r}{n^r},
\end{equation}
where $\beta$ is the non-negative root of $(n-1)S_1^2\beta^2=E$.
\end{theorem}

Note that the condition $S_r\leq S_1^r$ is needed in Theorem \ref{T1},
since power sums satisfy it by the positivity of $x_1,\dots,x_n$;
similarly, the condition $S_1^r\leq n^{r-1} S_r$
is necessary because of H\"older's inequality.
In the latter, equality holds if and only if $x_1=\cdots=x_n$,
which correponds to having $\alpha=0$ in \eqref{T1:eq001}.

Note also that
\[
E = (n-1) \Big(\sum_{i=1}^{n} x_i\Big)^2 - 2n\sum_{i<j} x_ix_j \leq (n-1)S_1^2.
\]
The inequality shows that $\beta\leq 1$ in Theorem \ref{T1}.
As a consequence, the factor before~$S_1^r$ in~\eqref{T1:eq003} is
never larger than $n^r$ (with the value $n^r$ being in fact attained
in the limit case $\beta=1$).
Therefore, \eqref{T1:eq003} is a refinement of the inequality $S_r\leq S_1^r$.

In order to study the maximum of $E$ a little more care is needed,
since the maximum is attained at points $(x_1,\dots,x_n)$
where some of the numbers can be zero.
To keep track of this, we introduce
a new quantity $\ntilde$ defined by the relation
\begin{equation}\label{def:ntilde}
\ntilde^{r-1} = S_1^r S_r^{-1}.
\end{equation}
The inequality $S_r\leq S_1^r$ shows that $\ntilde\geq 1$,
while H\"older's inequality gives $\ntilde\leq n$.

\begin{theorem}\label{T1max}
Let $r\in\mathbb{N}$, $r\geq 3$,
and let $S_1,S_r>0$ such that $S_r\leq S_1^r\leq n^{r-1}S_r$.
Define $\ntilde$ as in~\eqref{def:ntilde}.
Let $x_1,\dots,x_n\geq 0$ with $x_1+\cdots+x_n=S_1$, $x_1^r+\cdots+x_n^r=S_r$
and with potential energy $E$. Then,
\begin{equation}\label{T1max:eq001}
E \leq E^{max} = \frac{S_1^2}{\lceil\ntilde\rceil}
\Bigl(n (\lceil\ntilde\rceil-1) \alpha^2 + n-\lceil\ntilde\rceil\Bigr),
\end{equation}
where $\alpha$ is the root in $(-\frac{1}{\rule{0pt}{7pt}\lceil\ntilde\rceil-1},0)$ of \begin{equation}\label{T1max:eq002}
(1+\alpha(\lceil\ntilde\rceil-1))^r+(\lceil\ntilde\rceil-1)(1-\alpha)^r=\lceil\ntilde\rceil^rS_rS_1^{-r}.
\end{equation}
\end{theorem}

Unlike Theorem \ref{T1}, we do not have a converse statement for Theorem \ref{T1max}
similar to \eqref{T1:eq003}. Loosely speaking, such a converse result would give an estimate for $\ntilde$
by using~\eqref{T1max:eq002}.
However, $\ntilde$ is already present on the left-hand side of the equation
(whereas the coefficients of $\alpha$ in \eqref{T1:eq003} depend only on $n$).
Therefore, using \eqref{T1max:eq002} to estimate~$\ntilde$ appears to be uninteresting.

In the last section of the paper, we relate the potential energy to the discriminant
\[
\Delta = \prod_{i<j} (x_i-x_j)^2.
\]
By applying \eqref{eq:AG} to the numbers $(x_i-x_j)^2$, we immediately have
\begin{equation}\label{eq:EDelta}
\Bigg(\raisebox{2pt}{\ensuremath{\displaystyle\frac{\raisebox{-2pt}{$E$}}{\binom{n}{2}\rule{0pt}{10pt}}}}\Bigg)^{\binom{n}{2}} \geq \Delta.
\end{equation}
We improve this to the following.
\begin{theorem}\label{T2}
Let $S_1,S_2$ be positive real numbers such that
\begin{equation}\label{1401:eq001}
(n-1)S_2 < S_1^2 < nS_2.
\end{equation}
Let $x_1,\dots,x_n>0$ satisfy $x_1+\cdots+x_n=S_1$ and $x_1^2+\cdots+x_n^2=S_2$.
Denote by $E$ the potential energy of $x_1,\dots,x_n$ and by $\Delta$ their discriminant.
Then
\[
\Bigg(\raisebox{2pt}{\ensuremath{\displaystyle\frac{\raisebox{-2pt}{$E$}}{\binom{n}{2}\rule{0pt}{10pt}}}}\Bigg)^{\binom{n}{2}} \geq
\frac{(2n)^{\binom{n}{2}}}{Y(n)} \, \Delta,
\]
where $Y(n)$ is the hyperfactorial, i.e.~$Y(n)=2^23^3\cdots n^n$.
\end{theorem}

Using the identity $E=nS_2-S_1^2$, Theorem \ref{T2} can be restated as
\[
x_1^2+\cdots+x_n^2 \geq \frac{(x_1+\cdots+x_n)^2}{n} + 2\binom{n}{2}\left(\frac{\Delta}{Y(n)}\right)^{1/\binom{n}{2}}.
\]
This is reminiscent of another inequality of Hunter \cite[Theorem 1]{hunter2}
(see also \cite[Theorem 6.4.2]{cohen}), which states that in a field $F$
of degree $n$, there exists a non-rational algebraic integer $\alpha_1$ in $F$
with trace in $[0,n/2]$ and
\begin{equation}\label{eq:cohen}
(|\alpha_1|^2+\cdots+|\alpha_n|^2)\leq \frac{(\alpha_1+\cdots+\alpha_n)^2}{n} + \gamma_{n-1}\left(\frac{|\Delta_F|}{n}\right)^{1/(n-1)},
\end{equation}
where the $\alpha_i$ are the conjugates of $\alpha_1$;
$\Delta_F$ is the discriminant of the field $F$; and
$\gamma_{n-1}$ is the $(n-1)$th Hermite constant.
Inequality \eqref{eq:cohen} is proved by means of the geometry of numbers,
which is a very different method from the one we use to prove Theorem~\ref{T2}.

The proof of Theorem \ref{T2} is more similar in flavour to Schur's proof
of \cite[Satz II, Satz XI]{schur} and consists in maximizing the discriminant
by the method of Lagrange multipliers.
While Schur studies the discriminant in the full unit ball in $\RR^n$, we work
on the intersection of a sphere and a hyperplane.

Denote by $A(n)$ the factor in front of $\Delta$ in Theorem \ref{T2},
i.e.~$A(n)=Y(n)^{-1}(2n)^{\binom{n}{2}}$. Notice that $A(2)=1$ and that for $n\geq 2$ we have
\[
\frac{A(n+1)}{A(n)} = \frac{2^{n}}{(n+1)}\left(1+\frac{1}{n}\right)^{\frac{n^2-n}{2}} > 1.
\]
This shows that $A(n)>1$ for all $n>2$ and therefore Theorem \ref{T2}
is an improvement of \eqref{eq:EDelta} whenever $n>2$,
provided \eqref{1401:eq001} holds.
More precisely, for $n$ large we have
\[
\log A(n) \sim \frac{n^2}{2}\left(\log 2 - \frac{1}{2}\right) \approx \frac{n^2}{2} \times 0.1931..
\]
When $x_1,\dots,x_n$ are the roots of a monic separable polynomial in $\ZZ[x]$
it follows that, since $\Delta\geq 1$,
\begin{equation}\label{eq:1.21}
E \geq \frac{2}{\sqrt{e}} \binom{n}{2} + o(n^2)
\end{equation}
(and $2/\sqrt{e}\approx 1.21>1$). This is not as strong as other results available in the literature:
already in Hunter's paper it was derived $E\geq n^2(\sqrt{e}+o(1))$ for $n$ large enough. Nevertheless,
Theorem \ref{T2} gives a clean improvement of \eqref{eq:EDelta} valid for all degrees $n$,
under the only assumption \eqref{1401:eq001}.

It would be desirable to remove the condition \eqref{1401:eq001}
and have a cleaner statement in Theorem \ref{T2}.
In our proof, \eqref{1401:eq001} is used to infer
that none of $x_1,\dots,x_n$ can be zero. Because of this,
we find a simple differential equation related to $x_1,\dots,x_n$
and then use a recursive argument to pass from $n$ to $n-1$.
When one of the variables vanishes, the complexity in the recursive step
drastically increases and there seems not to be such a simple solution
as in the former situation.

We can do better than \eqref{eq:1.21} if we assume that, for all $i\neq j$, the numbers $(x_i-x_j)^2$
are distinct. If this is the case, we can apply Siegel's theorem \cite[Theorem II]{siegel}
and deduce, for $n$ large enough, the lower bound
\begin{equation}\label{eq:Esiegel}
E\geq \lambda\binom{n}{2},
\end{equation}
where $\lambda$ is any number strictly less than $\lambda_0=e(1+\vartheta^{-1})^{-\vartheta}=1.7336105..$,
with $\vartheta$ being the unique positive root of the transcendental equation
\[
(1+\vartheta)\log(1+\vartheta) - \frac{\log\vartheta}{1+\vartheta} = 1.
\]
We collect these thoughts in a proposition.
\begin{proposition}\label{P4}
Let $x_1,\dots,x_n>0$ be the roots of a
monic irreducible polynomial
and assume that, for all $i\neq j$,
the numbers $(x_i-x_j)^2$ are distinct. Then, for $n$ large enough,
the lower bound \eqref{eq:Esiegel} holds.
\end{proposition}

There are totally positive polynomials (i.e.~with only positive real roots)
for which the assumption in Proposition \ref{P4}
does not hold, one such example is the polynomial $((x-2)^2-2)((x-2)^2-3)$.
However, when the polynomial is irreducible and its trace is smaller than twice
the degree $n$, we found no counterexample for small $n$.
By a theorem of Luca \cite[Theorem 1.2]{luca}, the assumption holds
if the Galois group of the polynomial is $4$-transitive.
In particular, it is true if the Galois group is the full symmetric group.
Of the $896$ irreducible totally positive polynomials of degree $n<10$
and trace less than $2n$, $855$ of these have the full symmetric group
as their Galois group.

Even stronger bounds than \eqref{eq:Esiegel} can be obtained by the method
of auxiliary polynomials introduced by Smyth \cite{smyth},
once we observe that the potential energy is the trace of the polynomial with roots
$(x_i-x_j)^2$. The best constant in the trace problem has been obtained
with Smyth's method and is $1.793145$, due to Wang--Wu--Wu \cite{www}.

Finally, we conclude the paper by showing that Theorem \ref{T2}
can be extended to all the potentials of the form
\begin{equation}\label{class}
F(x_1,\dots,x_n) = \frac{a}{n}\sum_{i=1}^{n}x_i^2 + \frac{b}{n^2}\Bigl(\sum_{i=1}^nx_i\Bigr)^2 + \frac{c}{n}\sum_{i=1}^nx_i + d,
\end{equation}
with $a,b,c,d\in\RR$ and $a>0$. Clearly, the quadruple $(a,b,c,d)=(n^2,-n^2,0,0)$ recovers the potential energy $E$.
\begin{theorem}\label{T1.7}
Let $F$ be as in \eqref{class} and let $S_1,S_2>0$ such that \eqref{1401:eq001} holds.
Let $x_1,\dots,x_n$ be positive real numbers satisfying $x_1+\cdots+x_n=S_1$
and $x_1^2+\cdots+x_n^2=S_2$. Denote by $\Delta$ be the discriminant
of $x_1,\dots,x_n$. Then
\[
F(x_1,\dots,x_n)
\geq
\binom{n}{2}\frac{2a}{n}
\left(\frac{\Delta}{Y(n)}\right)^{1/\binom{n}{2}} + (a+b)\frac{S_1^2}{n^2} + \frac{cS_1}{n} +d.
\]
\end{theorem}

The paper is organised as follows: in Section \ref{S2} we sketch a proof of Proposition \ref{P1}
and of Corollary \ref{C1}. Then we move to the proofs of Theorem \ref{T1} and Theorem~\ref{T1max},
which span across Sections \ref{S3} and \ref{S4}. Finally, in Section \ref{S5}
we prove Theorem \ref{T2} and Theorem \ref{T1.7}.

\subsection*{Acknowledgements}
This work was supported by Czech Science Foundation GACR, grant 21-00420M,
the project PRIMUS/20/SCI/002 from Charles University
and Charles University Research Centre program UNCE/SCI/022.

\section{Considerations about Hunter's paper}\label{S2}

In his paper, Hunter found the maximum of the potential energy $E$
over all positive real numbers with fixed trace and norm \cite[Lemma~1]{hunter}.
In this section we sketch a proof of Proposition \ref{P1},
showing how the minimum of $E$ can be derived from Hunter's paper,
and of Corollary \ref{C1},
which essentially reverses the inequality between arithmetic mean and geometric mean.

Fix $s,p>0$ with $s^n>p$ and consider the set
\[
M = \{(x_1,\dots,x_n)\in\RR_+^n:\; x_1+\cdots+x_n=ns,\; x_1\cdots x_n=p\},
\]
which is an $(n-2)$-dimensional manifold in $\RR_+^n$.
As explained in \cite[p.150]{hunter}, by an appropriate choice
of variables we can construct a local chart from $M$ to a closed region $D$ in $\RR_+^{n-2}$
in such a way that the maximum and minimum of $E$ are interior points in $D$.
This implies that the minimum of $E$ on $M$ can be found among the critical points
of the function
\[
\sum_{1\leq i < j \leq n} (x_i-x_j)^2 + \lambda(x_1+\cdots+x_n) + \mu\log(x_1\cdots x_n).
\]
Differentiating with respect to a given variable $x_i$, we find that we must have
\[
x_i^2 + \left(\frac{\lambda}{2n}-s\right)x_i + \frac{\mu}{2n} = 0, \qquad i=1,\dots,n,
\]
and so $x_i$ satisfies a quadratic equation.
Therefore, at the point of minimum we must have $k$ of the $x_i$'s
equal in value to a first number $x$, and the remaining $n-k$
equal to a second number $y$, say.
Without loss of generality, we can assume $1\leq k\leq \lfloor n/2\rfloor$.
We obtain
\begin{equation}\label{2408:eq001}
\begin{cases}
kx+(n-k)y = ns\\
x^ky^{n-k} = p
\end{cases}
\end{equation}
and 
\[
E = k(n-k)(x-y)^2.
\]
Eliminating one variable using the trace condition in \eqref{2408:eq001}
and making the change of variable $y=(1-\alpha)s$, we deduce
\begin{equation}\label{eq002}
\frac{p}{s^n} = \left(1+\alpha\left(\frac{n}{k}-1\right)\right)^k (1-\alpha)^{n-k},
\end{equation}
where $\alpha\in(-\frac{k}{n-k},1)$.
Solving for $\alpha$, we find \cite[p.152]{hunter}
that for every $k$ there are exactly two solutions
$\alpha_1,\alpha_2$ of \eqref{eq002} in the desired interval,
and they satisfy
\[
-\frac{k}{n-k} < \alpha_1 < 0< \alpha_2 < 1.
\]
In addition, Hunter showed \cite[Lemma 2]{hunter} that we have
$\alpha_2\geq |\alpha_1|$.
In terms of $\alpha$, the energy is given by
\begin{equation}\label{2408:eq002}
E = \left(\frac{n}{k}-1\right)(ns)^2\alpha^2.
\end{equation}
Therefore, when seeking the minimum of $E$, it suffices to evaluate
\eqref{2408:eq002} when $\alpha=\alpha_1$ and $1\leq k\leq \lfloor n/2\rfloor$.
In analogy to \cite[Lemma 3]{hunter}, we prove that the smallest value
is obtained when $k=1$.
\begin{lemma}\label{S2:lemma}
Let $\alpha$ be the solution of \eqref{eq002} in $(-\frac{k}{n-k},0)$.
If we set
\(
u_k = \alpha^2\left(\frac{n}{k}-1\right),
\)
then
\begin{equation}\label{eq003}
\min_{k=1,\dots,\lfloor n/2\rfloor} u_k = u_1.
\end{equation}
\end{lemma}

\begin{proof}
Following Hunter's argument \cite[p.153--154]{hunter},
we make the change of variable $m=\frac{n}{k}-1$. After a few steps,
the problem reduces to showing that for $\alpha\in(-1/m,0)$ and
$1\leq m\leq n-1$, we have
\[
B(t) = \frac{1}{2}\left(t-\frac{1}{t}\right) - \log t <0,\quad \text{where } t=\frac{1+\alpha m}{1-\alpha}\in (0,1).
\]
Since $B$ is strictly increasing and $B(1)=0$, we obtain the lemma.
\end{proof}

To prove Proposition \ref{P1}, we combine \eqref{2408:eq002} and Lemma \ref{S2:lemma},
obtaining that the minimum of the potential energy on $M$
is $(n-1)(ns)^2\alpha^2$, where $\alpha$ is the root in $(-\frac{1}{n-1},0)$ of the equation
\begin{equation}\label{2408:eq003}
(1+\alpha(n-1)(1-\alpha)^{n-1} = s^{-n}p.
\end{equation}
For all the other points on $M$ we have $E\geq (n-1)(ns)^2\alpha^2$, as claimed.

As for the proof of Corollary \ref{C1}, denote by $E_0$ the above minimum
and let $(x_1,\dots,x_n)$ be a point on $M$, with energy $E\geq E_0$.
Since
\[
E_0=(n-1)(ns)^2\alpha^2
\quad\text{and}\quad
E=(n-1)(ns)^2\beta^2,
\]
with $\alpha,\beta$ taken in $(-\frac{1}{n-1},0)$, it follows $\beta\leq \alpha$.
The function
\[
f(\beta) = \frac{1}{(1+\beta(n-1))(1-\beta)^{n-1}}
\]
is strictly decreasing in the interval $(-\frac{1}{n-1},0)$
since $f'(\beta)<0$ in this range. Thus
\[
\frac{s^n}{p} = f(\alpha) \leq f(\beta),
\]
which gives the corollary.

\section{Potential energy and power sums}\label{S3}

We turn now to Theorems \ref{T1} and \ref{T1max}.
In this section we do an initial analysis of the potential energy
by using the method of Lagrange multipliers,
which allows us to pass to a one-dimensional problem.
After that, we prove a number of auxiliary lemmas
that will be useful to solve such a problem.
The proofs of the theorems will then be completed in Section \ref{S4}.

Fix $S_1,S_r$ as in Theorem~\ref{T1} and Theorem~\ref{T1max}
and let $M_{r,n}$ be the set of points $(x_1,\dots,x_n)\in\RR_{\geq 0}^n$ satisfying
\begin{equation}\label{2508:eq001}
x_1+\cdots+x_n = S_1,\quad x_1^r+\cdots+x_n^r = S_r.
\end{equation}
Our goal is to maximize (and minimize) the energy $E=E(x_1,\dots,x_n)$ over $M_{r,n}$.
If the point of maximum is on $M_{r,n}\cap\partial\RR^n_{\geq 0}$, then
at least one of $x_1,\dots,x_n$ vanishes. More precisely,
there will be an integer $j$, with $1\leq j\leq n-1$,
such that $x_1,\dots,x_{n-j}\neq 0$ and $x_{n-j+1}=\cdots=x_n=0$.
Note that
\[
E(x_1,\dots,x_{n-j},0,\dots,0)= E(x_1,\dots,x_{n-j}) + j \sum_{i=1}^{n-j}x_i^2.
\]
Recalling from the introduction the identity $E=nS_2-S_1^2$,
we obtain
\begin{equation}\label{2909:eq004}
E(x_1,\dots,x_{n-j},0,\dots,0)= \frac{n}{n-j} E(x_1,\dots,x_{n-j}) + \frac{jS_1^2}{n-j}.
\end{equation}
Since $S_1$ is fixed, the problem reduces to finding the maximum of $E$
over $M_{r,n-j}$, with the additional information that at the point of
maximum we have $x_i>0$ for every $i$. An identical argument works for the minimum of $E$.

From now on, let us assume that the point of maximum and the point of minimum
are in $M_{r,n}\cap\RR^{n}_{+}$.
In Section~\ref{S4} we will return to the case when we have zeros and
will use \eqref{2909:eq004} to determine the global maximum
and the global minimum of $E$.

Furthermore, if $S_1^r=n^{r-1}S_r$, then $M_{r,n}$ reduces to the single point
$x_1=\cdots=x_n=n^{-1}S_1$, in which case $E=0$ and there is nothing to prove.
We assume therefore $S_1^r<n^{r-1}S_r$ and
search for the extremal values of $E$ by the method of Lagrange multipliers.

\subsection{Lagrange multipliers}
The functions
\[
F(x_1,\dots,x_n)=x_1+\cdots+x_n
\quad \text{and}\quad
G(x_1,\dots,x_n)=x_1^r+\cdots+x_n^r
\]
have gradients
\[
\begin{split}
\nabla F = (1,\dots,1),\quad \nabla G=r(x_1^{r-1},\dots,x_n^{r-1}).
\end{split}
\]
Since $r>1$, $\nabla F$ and $\nabla G$ are linearly independent unless
$x_1=x_2=\cdots=x_n$, which is excluded on $M_{r,n}$ because we are assuming
$S_1^r<n^{r-1}S_r$.
Therefore, by the Lagrange multipliers theorem,
the maximum and minimum of the potential energy
on $M_{r,n}$ can be found among the critical points of the function
\[
\sum_{1\leq i<j\leq n} (x_i-x_j)^2 + \lambda(x_1+\cdots+x_n) + \mu(x_1^r+\cdots+x_n^r).
\]
Differentiating with respect to each variable gives
\begin{equation}\label{2508:eq002}
2n(x_i-S_1) + \lambda + r\mu x_i^{r-1} = 0,\qquad i=1,\dots,n.
\end{equation}

The above is a polynomial equation of degree $r-1$
and is non-degenerate (i.e.~$\mu\neq 0$) since otherwise
$x_1,\dots,x_n$ would all be equal to a single value, which we have excluded.
More precisely, we know that \eqref{2508:eq002} must have at least two
distinct positive solutions.
At the same time, by Descarte's rule of signs, we know that trinomials
can have at most two positive roots. Hence, the equation in \eqref{2508:eq002}
has exactly two positive solutions.

In other words, if $x_i$ is a root of \eqref{2508:eq002},
then it equals one of two positive values, say $x$ or $y$.
We can therefore write, for some integer $1\leq k\leq n-1$,
\begin{equation}\label{2608:eq003}
\begin{cases}
S_1 = kx+(n-k)y,\\
S_r = kx^r+(n-k)y^r
\end{cases}
\end{equation}
and
\[
E = k(n-k)(x-y)^2.
\]
Without loss of generality, we can assume $k\leq \lfloor n/2\rfloor$.
Making the change of variable $ny=(1-\alpha)S_1$, with $\alpha\in(-\frac{k}{n-k},1)$,
and solving for $x$ in the first equation of \eqref{2608:eq003}, we obtain
\begin{equation}\label{2011:eq001}
\begin{cases}
nx = S_1(1+\alpha(\frac{n}{k}-1)),\\
k(1+\alpha(\frac{n}{k}-1))^r + (n-k)(1-\alpha)^r = \displaystyle\frac{n^rS_r}{S_1^r}
\end{cases}
\end{equation}
and
\begin{equation}\label{2011:eq002}
E = \left(\frac{n}{k}-1\right) \alpha^2 S_1^2.
\end{equation}

Solutions $\alpha$ of the second equation in \eqref{2011:eq001}
will give extremal values of the energy.
In general, we would perhaps expect to find two solutions
corresponding to a local maximum and a local minimum.
Of course, sometimes we could have only one of the two (or none).
In Lemma~\ref{S2:lemma1} below we make this argument precise.

Set
\begin{equation}\label{def:g}
g(\alpha) := k(1+\alpha(\tfrac{n}{k}-1))^r + (n-k)(1-\alpha)^r - \frac{n^rS_r}{S_1^r}.
\end{equation}
Notice that $g(0)=n(1-n^{r-1}S_rS_1^{-r})<0$ since $S_1^r<n^{r-1}S_r$. Also,
\begin{equation}\label{2909:eq005}
g'(\alpha) = (n-k)r((1+\alpha(\tfrac{n}{k}-1)^{r-1}-(1-\alpha)^{r-1}),
\end{equation}
which shows that $g$ is strictly decreasing in $(-\frac{k}{n-k},0)$
and strictly increasing in $(0,1)$.
In particular, we can have at most one solution of the equation $g(\alpha)=0$
in each of the two intervals. Depending on the location of $k$
relative to the number $\tilde{n}$ defined in~\eqref{def:ntilde},
we can find solutions of $g(\alpha)=0$ to the left or to the right of zero.

\begin{lemma}\label{S2:lemma1}
The equation $g(\alpha)=0$ has a solution $\alpha_1$ in $(-\frac{k}{n-k},0)$
if and only if $k>n-\ntilde$.
Similarly, there is a solution $\alpha_2\in(0,1)$ if and only if $k<\ntilde$.
For those $k$ such that both $\alpha_1$ and $\alpha_2$ are defined,
we have $|\alpha_1|\geq \alpha_2$.
\end{lemma}

\begin{proof}
Since $g$ is decreasing to the left of zero
and $g(0)<0$, we deduce that $g$ has a root $\alpha_1$ in $(-\frac{k}{n-k},0)$
if and only if $g(-\frac{k}{n-k})>0$, which is equivalent to $k>n-\ntilde$.
Similarly, since $g$ is increasing to the right of zero,
we deduce that $g$ has a root $\alpha_2$ in $(0,1)$
if and only if $g(1)>0$, which is equivalent to $k<\ntilde$.

Assume now that both $\alpha_1$ and $\alpha_2$ are defined.
Since $g$ is increasing in $(0,1)$,
the last claim will follow if we can show $g(\alpha_2)\leq g(-\alpha_1)$.
By the identity $g(\alpha_2)=g(\alpha_1)$,
this amounts to proving $g(\alpha_1)\leq g(-\alpha_1)$.
In other words, we want to prove that, for $\alpha$ in $(-\tfrac{k}{n-k},0)$, we have
\[
H(\alpha):= (1+\alpha m)^r + m(1-\alpha)^r - (1-\alpha m)^r - m(1+\alpha)^r \leq 0,
\]
where we use the shorthand $m=\frac{n}{k}-1$ ($m\geq 1$ since $k\leq n/2$).
Notice that $H(0)=0$ and
\[
\frac{H'(\alpha)}{mr} = (1+\alpha m)^{r-1}+(1-\alpha m)^{r-1}-(1-\alpha)^{r-1}-(1+\alpha)^{r-1}.
\]
Expanding the powers, the above gives
\[
\frac{H'(\alpha)}{mr} = 2 \sum_{0<2j\leq r-1} \binom{r-1}{2j} \alpha^{2j}(m^{2j}-1) \geq 0.
\]
Therefore $H$ is non-decreasing and $H(\alpha)\leq 0$ for $\alpha\leq 0$, as desired.
\end{proof}

\subsection{Auxiliary lemmas}
As we saw in \eqref{2011:eq002}, the potential energy
can be expressed directly in terms of $\alpha$; its extremal values
are obtained by picking $\alpha$ as one of the roots $\alpha_1,\alpha_2$ of the equation $g(\alpha)=0$.
Depending on the choice of the root,
we get two distinct functions that depend indirectly
on $k$ and $n$ through $\alpha_1,\alpha_2$. Set
\begin{equation}\label{0209:eq001}
U(k,n) := \alpha_1^2\left(\frac{n}{k}-1\right),
\quad
V(k,n) := \alpha_2^2\left(\frac{n}{k}-1\right).
\end{equation}
In the remainder of this section we study the behaviour of the functions $U$ and $V$.
First, we show that they are both monotonic functions of $k$.

\begin{lemma}\label{S2:lemma2}
The function $U$ is decreasing in $k$.
The function $V$ is increasing in $k$.
\end{lemma}

\begin{proof}
As in Lemma \ref{S2:lemma1}, we make the change of variable
$m=\frac{n}{k}-1$, which corresponds to $k=\frac{n}{m+1}$
and $n-k=\frac{nm}{m+1}$. To prove that $U$ is decreasing in $k$
we will show that, as a function of $m\in[1,n-1]$, the quantity
\[
U = \alpha_1^2 m
\]
has positive first derivative.
For the sake of keeping the notation light,
let us temporarily write $\alpha$ instead of $\alpha_1$.
First observe that
\begin{equation}\label{0509:eq002}
\frac{dU}{dm} = \alpha^2 + 2\alpha m \frac{d\alpha}{dm} = \alpha\left(\alpha+2m\frac{d\alpha}{dm}\right)
>0 \iff
\alpha+2m\frac{d\alpha}{dm} < 0.
\end{equation}
Next we consider the identity $g(\alpha)=0$. In terms of $m$ it reads
\[
\frac{1}{m+1}(1+\alpha m)^r + \frac{m}{m+1}(1-\alpha)^r = \frac{n^{r-1}S_r}{S_1^r}.
\]
Differentiating in $m$ we obtain
\[
-\frac{(1+\alpha m)^r}{(m+1)^2} + \frac{r(1+\alpha m)^{r-1}}{m+1}\left(\alpha+m\frac{d\alpha}{dm}\right)
+\frac{(1-\alpha)^r}{(m+1)^2} - \frac{rm(1-\alpha)^{r-1}}{m+1} \frac{d\alpha}{dm} = 0.
\]
Solving for $\frac{d\alpha}{dm}$ gives
\[
m\frac{d\alpha}{dm} = \frac{(1+\alpha m)^{r-1}(1+\alpha m-\alpha r(m+1)) - (1-\alpha)^r)}{r(m+1)((1+\alpha m)^{r-1}-(1-\alpha)^{r-1})}.
\]
Using this and setting $t=(1+\alpha m)(1-\alpha)^{-1}$,
we deduce that
\begin{equation}\label{0509:eq001}
\alpha + 2m \frac{d\alpha}{dm} =
\frac{(2-r)t^r+rt^{r-1}-rt-(2-r)}{r(t+m)(t^{r-1}-1)}.
\end{equation}
Since $\alpha_1\in(-\frac{1}{m},0)$, we see that $t\in(0,1)$
and so the above fraction is well defined.
Let $B(t)$ be the numerator in \eqref{0509:eq001}.
We have $B(1)=0$ and, for $t>0$,
\[
B'(t) = -r(t-1)^2 (1+2t+3t^2+\cdots+(r-2)t^{r-3}) < 0
\]
(the factorization can be proved by induction on $r$),
which shows $B(t)>0$ for $t\in(0,1)$.
Combining \eqref{0509:eq002} and \eqref{0509:eq001} we deduce $\frac{dU}{dm}>0$,
as claimed.

As for the function $V$, we argue similarly, except that now $\alpha_2\in(0,1)$
and therefore $t\in(1,\infty)$. Since $B(t)<0$ in this range,
combining again \eqref{0509:eq002} and \eqref{0509:eq001} we deduce $\frac{dV}{dm}<0$,
which implies that $V$ is increasing in $k$.
\end{proof}

Because of the discussion at the beginning of the section, we will also need
to compare the functions $U(k,n)$ and $V(k,n)$ for different values of $n$.
To be precise, we need a slight modification of these functions.
Set
\begin{equation}\label{def:FG}
F(k,n) = \frac{V(k,n)+1}{n}
\quad\text{and}\quad
G(k,n) = \frac{U(k,n)+1}{n}.
\end{equation}
In Lemma \ref{S2:lemma3} below we show that $F$ is monotonic in $n$.
Later, in Lemma \ref{S2:lemma4}, we show that $G$ is monotonic
along diagonal lines in the $(k,n)$-plane.

\begin{lemma}\label{S2:lemma3}
The function $F(k,n)$ defined above is decreasing in $n$.
\end{lemma}

\begin{proof}
First, recalling the definition of $V$ in \eqref{0209:eq001}, we can write
\begin{equation}\label{2809:eq001}
F(k,n) = \frac{1}{n} + \left(\frac{1}{k}-\frac{1}{n}\right)\alpha^2,
\end{equation}
where $\alpha=\alpha(n)$ is the positive root of $g(\alpha)=0$
(and $g$ is defined in \eqref{def:g}).
Let us show that $\frac{dF}{dn}<0$. Differentiating in \eqref{2809:eq001} gives
\[
\frac{dF}{dn}= -\frac{1}{n^2}+\frac{\alpha^2}{n^2}+\left(\frac{1}{k}-\frac{1}{n}\right)2\alpha\frac{d\alpha}{dn}<0.
\]
Next we use the identity $g(\alpha)=0$ to obtain an expression for $\frac{d\alpha}{dn}$.
Since
\[
\frac{1}{n^r}\Bigl(k(1+(\tfrac{n}{k}-1)\alpha)^r+(n-k)(1-\alpha)^r\Bigr)=\frac{S_r}{S_1^r},
\]
if we differentiate in $n$ and solve for $\frac{d\alpha}{dn}$ we get
\[
\frac{d\alpha}{dn}
=
\frac{
     r(k(1+(\frac{n}{k}-1)\alpha)^r+(n-k)(1-\alpha)^r)-nr\alpha(1+(\frac{n}{k}-1)\alpha)^{r-1}-n(1-\alpha)^r
     }{
     nr(n-k)((1+(\frac{n}{k}-1)\alpha)^{r-1}-(1-\alpha)^{r-1})
     }.
\]
Setting $t=(1-\alpha)^{-1}(1+(\frac{n}{k}-1)\alpha)$ and observing that $t>1$ (since $\alpha\in(0,1)$),
it follows that
\[
\frac{dF}{dn}
=
-\frac{r(t^{r-1}-1)-2(r-1)(t-1)}{rk^2(t^{r-1}-1)(t+\frac{n}{k}-1)^2}.
\]
Now consider the function $B(t)=r(t^{r-1}-1)-2(r-1)(t-1)$. We have $B(1)=0$ and
\[
\frac{B'(t)}{r-1} = rt^{r-2} - 2 > 0
\]
since $t>1$ and $r\geq 3$. Therefore $\frac{dF}{dn}<0$ and $F$ is decreasing, as claimed.
\end{proof}

\begin{lemma}\label{S2:lemma4}
Let $k,n$ be fixed and let $G(k,n)$ be the function defined in \eqref{def:FG}.
For $\sigma\in[0,1]$, the function $G(k+\sigma,n+\sigma)$
is decreasing in $\sigma$.
In particular, we have
\[
G(k,n)\geq G(k+1,n+1).
\]
\end{lemma}

\begin{proof}
With a slight abuse of notation, we set $G(\sigma)=G(k+\sigma,n+\sigma)$.
Let us show that $\frac{dG}{d\sigma}<0$, from which the lemma follows.
By definition we have
\begin{equation}\label{1110:eq001}
G(\sigma) = \frac{1}{n+\sigma}\left(1+\frac{n-k}{k+\sigma}\alpha^2\right),
\end{equation}
where $\alpha=\alpha(\sigma)$ is the root in $(-\frac{k+\sigma}{n-k},0)$ of the equation
\begin{equation}\label{1110:eq002}
\frac{(k+\sigma)}{(n+\sigma)^r}\left(1+\frac{n-k}{k+\sigma}\alpha\right)^r
+
\frac{(n-k)}{(n+\sigma)^r}(1-\alpha)^r = \frac{S_r}{S_1^r}.
\end{equation}
From \eqref{1110:eq001} we have
\[
\frac{dG}{d\sigma}
=
\frac{(n-k)}{(n+\sigma)(k+\sigma)}
\left(
2\alpha\frac{d\alpha}{d\sigma}
-
\alpha^2\left(\frac{1}{k+\sigma}+\frac{1}{n+\sigma}\right)
\right)
-
\frac{1}{(n+\sigma)^2}.
\]
Differentiating in \eqref{1110:eq002} with respect to $\sigma$
we also have
\begin{multline*}
\frac{(n+\sigma)-r(k+\sigma)}{r(n-k)}\left(1+\frac{n-k}{k+\sigma}\alpha\right)^r
-\frac{\alpha (n+\sigma)}{(k+\sigma)}\left(1+\frac{n-k}{k+\sigma}\alpha\right)^{r-1}
\\
+
(n+\sigma)\left(1+\frac{n-k}{k+\sigma}\alpha\right)^{r-1}\frac{d\alpha}{d\sigma}
-(1-\alpha)^r
-(n+\sigma)(1-\alpha)^{r-1}\frac{d\alpha}{d\sigma} = 0.
\end{multline*}
Solving for $\frac{d\alpha}{d\sigma}$ and making the change of variable $t=(1-\alpha)^{-1}(1+\frac{n-k}{k+\sigma}\alpha)$ gives
\[
\frac{d\alpha}{d\sigma} = \frac{t^r(n+\sigma)(r-1)-t^{r-1}(r(n-k)+r(n-k)}{r(n-k)(k+\sigma)(t^{r-1}-1)(t+\frac{n-k}{k+\sigma})}
\]
Note that, since $\alpha\in(-\frac{k+\sigma}{n-k},0)$, we have $t\in(0,1)$, so the above fraction is well defined.
Going back to $\frac{dG}{d\sigma}$ we obtain
\begin{equation}\label{1110:eq003}
\frac{dG}{d\sigma} = \frac{(r-2)t^{r+1}-2(r-1)t^r+t^2r}{r(k+\sigma)^2(t^{r-1}-1)(t+\frac{n-k}{k+\sigma})^2}.
\end{equation}
The numerator vanishes when $t=0,1$. Set $B(t)=(r-2)t^{r-1}-2(r-1)t^{r-2}+r$.
Then $B(1)=0$ and for $0<t<1$ we have
\[
\frac{B'(t)}{(r-1)(r-2)t^{r-3}} = t-2 < 0,
\]
which shows $B(t)>0$ for $t\in(0,1)$.
Therefore the numerator in \eqref{1110:eq003} is always positive.
Since the denominator is negative for $t\in (0,1)$, we deduce $\frac{dG}{d\sigma}<0$,
which gives the lemma.
\end{proof}

\section{Proofs of Theorem \ref{T1} and Theorem \ref{T1max}}\label{S4}

We are now ready to combine the results from the previous section
in order to finish the proofs of Theorems \ref{T1} and \ref{T1max}.
Loosely speaking, with the Lagrange multipliers method
we determined extremal points of the energy in the set $M_{r,n}\cap\RR_{+}^n$.
Using the auxiliary lemmas,
we will select the largest local maximum and the smallest local minimum
over such a set. Then, by comparing with the values of the energy
on the boundary $M_{r,n}\cap \partial\RR_{\geq 0}^n$ and recalling the discussion
at the beginning of the previous section, we will determine the global maximum and the global minimum of $E$.

\subsection*{Proof of Theorem \ref{T1}}
Let us determine the smallest minimum of $E$ in the set $M_{r,n}\cap\RR_{+}^n$.
From the analysis in the previous section,
such a minimum is to be found among the values
\[
\{S_1^2U(k,n),\;S_1^2V(k,n):\;1\leq k\leq \lfloor n/2\rfloor\}.
\]

We claim that the smallest value in the above set is $S_1^2V(1,n)$.
To see this, we distinguish in two cases according to the size of $\ntilde$.

If $\ntilde>n-\lfloor n/2\rfloor$, then by Lemma \ref{S2:lemma1}
we know that $U(k,n)$ is well-defined for $k$ in the
range $n-\ntilde<k\leq\lfloor{n/2}\rfloor$,
while $V(k,n)$ is well defined for $1\leq k<\lfloor{n/2}\rfloor$.
Using the monotonicity properties of $U$ and~$V$ proved in Lemma \ref{S2:lemma2},
we can write, for $1\leq k\leq \lfloor n-\ntilde\rfloor< k'\leq\lfloor{n/2}\rfloor$,
\begin{equation}\label{2909:eq001}
V(1,n) \leq V(k,n) \leq V(\lfloor{n-\ntilde}\rfloor,n) \leq V(k',n) \leq U(k',n) \leq U(\lfloor{n-\ntilde}\rfloor+1,n),
\end{equation}
from which we see that $V(1,n)$ is the smallest value.

If instead $1<\ntilde\leq n-\lfloor n/2\rfloor$, Lemma \ref{S2:lemma1}
gives no value of $k\leq \lfloor n/2\rfloor$ for which $U$ is defined,
while $V$ is defined for $1\leq k<\ntilde$.
Again by the monotonicity of $V$, we have
\begin{equation}\label{2909:eq002}
V(1,n) \leq V(2,n)\leq \cdots \leq V(\lceil{\ntilde}\rceil-1,n),
\end{equation}
so again $V(1,n)$ is the smallest value, which proves our claim.

Now we want to show that $S_1^2V(1,n)$ is in fact the global minimum
of $E$ over $M_{r,n}$. To do this, we need to compare it with the smallest
values attained by $E$ on the boundary $M_{r,n}\cap\partial\RR_{\geq 0}^n$.
By \eqref{2909:eq004} we know that, up to a 
multiple of $S_1$ and up to an explicit factor,
the minimum on the boundary will be of the
form $S_1^2V(1,n-j)$ for some $j$ with $1\leq j\leq n-1$.
This means that we aim to prove
\[
S_1^2V(1,n) \leq \frac{n}{n-j}S_1^2V(1,n-j) + \frac{jS_1^2}{n-j}.
\]
Rearranging, we want
\[
F(1,n) = \frac{V(1,n)+1}{n} \leq \frac{V(1,n-j)+1}{n-j} = F(1,n-j).
\]
By Lemma \ref{S2:lemma3} we know that $F$ is decreasing in $n$,
and therefore the above holds. We conclude that
\[
E_{min} = S_1^2V(1,n) = (n-1) S_1^2 \alpha^2,
\]
where $\alpha\in(0,1)$ is the root of $g(\alpha)=0$,
where $g$ is defined in \eqref{def:g}.
This gives the first part of Theorem \ref{T1}.

To prove the second part of the theorem, let be $(x_1,\dots,x_n)$ be
any point on $M_{r,n}$, and let $E$ be its potential energy.
In particular, $E\geq E_{min}$. If $\alpha,\beta$
are the non-negative roots of 
\[
E=(n-1)S_1^2\beta^2,\qquad E_{min} = (n-1)S_1^2\alpha^2,
\]
then we must have $0\leq \alpha\leq \beta<1$.
Furthermore, the function
\[
f(\alpha)=(1+(n-1)\alpha)^r+(n-1)(1-\alpha)^r
\]
is increasing in $(0,1)$, see \eqref{2909:eq005}. Thus, we deduce
\[
\pushQED{\qed} 
\frac{n^rS_r}{S_1^r} = f(\alpha) \leq f(\beta).
\qedhere\popQED
\]

Concerning the proof of Theorem \ref{T1max}, we will need more care 
since the global maximum of $E$ will depend on the number $\ntilde$
defined in~\eqref{def:ntilde}.

\subsection*{Prooof of Theorem~\ref{T1max}}
We proceed as in the proof of Theorem~\ref{T1} and start by
looking at the critical points of $E$
with only positive coordinates.
In this case, the maximal value is found among the numbers
\begin{equation}\label{2210:eq003}
\{S_1^2U(k,n),\;S_1^2V(k,n):\;1\leq k\leq \lfloor n/2\rfloor\}.
\end{equation}
Let us set $k_*=\lfloor n-\ntilde\rfloor=n-\lceil\ntilde\rceil$.
In other words, $k_*$ is the unique integer such that $n-k_*-1<\ntilde< n-k_*$.
By H\"older's inequality, this implies that if $x_1+\cdots+x_n=S_1$
and $x_1^r+\cdots+x_n^r=S_r$, then at most $k_*$ of the numbers $x_1,\dots,x_n$
can be zero.

As in the proof of Theorem \ref{T1}, we distinguish
two cases depending on the size of~$\ntilde$.
If $\ntilde>n-\lfloor n/2\rfloor$,
then by repeating the argument before \eqref{2909:eq001} we deduce
that the largest value in \eqref{2210:eq003}
is $S_1^2U(k_*+1,n)$, which gives therefore the
maximum of $E$ in $M_{r,n}\cap\RR_+^n$.
If we look at the boundary, $M_{r,n}\cap\partial\RR_{\geq 0}^n$,
where we can have zeros,
then by \eqref{2909:eq004} we want to compare $U(k_*+1,n)$ with the quantity
\begin{equation}\label{2210:eq001}
\frac{n}{n-j}U(k_*+1-j,n-j) + \frac{j}{n-j},\quad 1\leq j\leq k_*.
\end{equation}
We claim that, for any $j$, \eqref{2210:eq001} is greater than $U(k_*+1,n)$.
After rearranging, this translates into the inequality
\[
G(k_*+1,n)=\frac{U(k_*+1,n)+1}{n} \leq \frac{U(k_*+1-j,n-j)+1}{n-j} = G(k_*+1-j,n-j).
\]
By Lemma~\ref{S2:lemma4} we know that $G(k,n)\geq G(k+1,n+1)$.
Applying this $j$ times starting from $G(k_*+1-j,n-j)$, we obtain the claim.
Moreover, since \eqref{2210:eq001} is largest when~$j$~is as large as possible,
by taking $j=k_*$ we deduce
\begin{equation}\label{2210:eq006}
E^{max} = \frac{S_1^2n}{n-k_*}U(1,n-k_*) + \frac{S_1^2k_*}{n-k_*}.
\end{equation}

Now, if $\ntilde\leq n-\lfloor n/2\rfloor$, the set in \eqref{2210:eq003}
contains only the elements $S_1^2V(k,n)$, with $1\leq k<\ntilde$,
and they are in increasing order (cf.~\eqref{2909:eq002}).
Therefore, the largest element is $S_1^2V(\lceil \ntilde\rceil-1,n)$.
Let us compare this with the value of the energy
on $M_{r,n}\cap\partial\RR_{\geq 0}^n$, where we can have zeros.
If we have $j$ zeros, with
\[
\ntilde\leq (n-j) - \lfloor(n-j)/2\rfloor
\]
(that is, $j\leq j_*=n-\lceil2\ntilde\rceil-\eps$, where $\eps=0,1$ according to whether $\lceil2\ntilde\rceil$ is even or odd, respectively),
then, by \eqref{2909:eq004} and what we have just discussed,
we need to compare $S_1^2V(\lceil\ntilde\rceil-1,n)$ with
\begin{equation}\label{2210:eq004}
\frac{S_1^2 n}{n-j} V(\lceil \ntilde\rceil-1,n-j) + \frac{S_1^2j}{n-j}.
\end{equation}
Lemma \ref{S2:lemma3} tells us that
\[
\frac{V(k,n)+1}{n} \leq \frac{V(k,n-j)+1}{n-j}.
\]
Applying this with $k=\lceil\ntilde\rceil-1$ we deduce that \eqref{2210:eq004}
is greater than $S_1^2V(\lceil \ntilde\rceil-1,n)$ and attains the largest value
when $j=j_*$. In other words, the quantity in \eqref{2210:eq004} with $j=j_*$
is a candidate for the global maximum of the potential energy,
but we still need to compare it with the maximal values
of the energy when we allow $j>j_*$ zeros.
As soon as $j>j_*$, the function $U$ starts to appear.
By Lemma \ref{S2:lemma1}, we know that $U$ is generally larger than $V$.
Because of this, we may want to compare
\begin{equation}\label{2111:eq005}
\frac{S_1^2 n}{n-j_*} V(\lceil \ntilde\rceil-1,n-j_*) + \frac{S_1^2j_*}{n-j_*}.
\end{equation}
with
\begin{equation}\label{2210:eq005}
\frac{S_1^2n}{n-j_*-1} U(\lceil \ntilde\rceil-1,n-j_*-1) + \frac{S_1^2(j_*+1)}{n-j_*-1}.
\end{equation}
We claim that the above is larger than \eqref{2111:eq005}.
To see this, apply Lemma \ref{S2:lemma3} to bound
\[
\frac{V(\lceil \ntilde\rceil-1,n-j_*)+1}{n-j_*}
\leq
\frac{V(\lceil \ntilde\rceil-1,n-j_*-1)+1}{n-j_*+1},
\]
and then use Lemma \ref{S2:lemma1} to bound
\[
\frac{V(\lceil \ntilde\rceil-1,n-j_*-1)+1}{n-j_*+1}
\leq
\frac{U(\lceil \ntilde\rceil-1,n-j_*-1)+1}{n-j_*+1},
\]
which implies our claim. Furthermore, observe that \eqref{2210:eq005}
can be written as
\[
\frac{S_1^2n}{n-j_*-1} U(\lfloor n-j_*-1-\ntilde\rfloor+1,n-j_*-1) + \frac{S_1^2(j_*+1)}{n-j_*-1}.
\]
At this point we have the same type of quantity as in \eqref{2210:eq001}
and we can argue as in \eqref{2210:eq001}--\eqref{2210:eq006}, showing that the largest
value of the energy is
\begin{equation}\label{2210:eq007}
E^{max} = \frac{S_1^2n}{n-k_*}U(1,n-k_*) + \frac{S_1^2k_*}{n-k_*}.
\end{equation}
Combining \eqref{2210:eq006} and \eqref{2210:eq007} gives
Theorem \ref{T1max}.
\qed

\section{Energy and discriminant}\label{S5}

In this section we prove Theorem \ref{T2}, which relates the potential
energy and the discriminant $\Delta$. At the end of the section we explain
how to modify the argument to obtain Theorem \ref{T1.7}.

Let $S_1,S_2>0$ with $(n-1)S_2<S_1^2<nS_2$.
Our strategy consists in maximizing $\log\Delta$ over the manifold
$M\subseteq \RR_{\geq 0}^n$ of non-negative real points $x_1,\dots,x_n$ such that
\[
x_1+\cdots+x_n = S_1,
\quad
x_1^2+\cdots+x_n^2 = S_2.
\]
Note that we can assume that $\Delta>0$, for otherwise
Theorem \ref{T2} holds trivially. Hence, we can assume
that $x_i\neq x_j$ for $i\neq j$.
We argue as in Section \ref{S3}: the functions
$F(x_1,\dots,x_n)=x_1+\cdots+x_n$ and
$G(x_1,\dots,x_n)=x_1^2+\cdots+x_n^2$
have gradients
\[
\begin{split}
\nabla F = (1,\dots,1),\quad \nabla G=2(x_1,\dots,x_n).
\end{split}
\]
Therefore, $\nabla F$ and $\nabla G$ are linearly independent unless
$x_1=x_2=\cdots=x_n$, which is excluded on $M$ since we are assuming $S_1^2<nS_2$.
Note also that the assumption $(n-1)S_2<S_1^2$ implies
that none of the points $x_1,\dots,x_n$ can be zero.
In particular, this holds for the point of maximum.

By the Lagrange multipliers theorem, the maximum of $\log\Delta$
on $M$ can be found among the critical points of the function
\begin{equation}\label{3108:eq010}
\frac{1}{2}\log \Delta - \frac{\lambda}{2}(x_1^2+\cdots+x_n^2) + \mu(x_1+\cdots+x_n).
\end{equation}
Let $(x_1,\dots,x_n)$ be one such critical point and write
\[
f(x) = \prod_{i=1}^{n}(x-x_i) = x^n + c_{n-1}x^{n-1}+\cdots + c_1x + c_0.
\]
A calculation shows that (cf.~\cite[(7)--(8)]{siegel}),
after differentiation with respect to $x_k$ in~\eqref{3108:eq010}, we obtain
\[
\frac{f''}{f'}(x_k) - \lambda x_k + \mu = 0,\quad k=1,\dots,n.
\]
Clearing denominator, we see that the polynomial
$f''(x) - (\lambda x-\mu)f'(x)$ vanishes for all $x=x_1,\dots,x_n$.
Since the $x_i$'s are all distinct, if follows that such
a polynomial must be a multiple of $f$.
Comparing the leading coefficients, we deduce the identity
\begin{equation}\label{3108:eq011}
f''(x) - (\lambda x-\mu)f'(x) + \lambda n f(x) = 0.
\end{equation}
In particular, by looking at the coefficient of $x^k$,
we deduce the three-term recurrence relation
\[
\lambda(n-k)c_k = -\mu(k+1)c_{k+1} - (k+2)(k+1)c_{k+2}.
\]
For $k=n-1$ and $k=n-2$ we obtain
\begin{equation}\label{3108:eq012}
c_{n-1} = -n\mu \lambda^{-1},
\qquad
c_{n-2} = \binom{n}{2}(\mu^2-\lambda)\lambda^{-2}.
\end{equation}
Consequently,
\begin{equation}\label{3108:eq015}
E = (n-1)c_{n-1}^2-2nc_{n-2} = \binom{n}{2}\frac{2n}{\lambda}.
\end{equation}
Now we evaluate the maximum of the discriminant. In order to do this,
we work first with the resultant and then go back to the discriminant
by means of the identity
\begin{equation}\label{3108:eq013}
\Delta=(-1)^{\frac{n(n-1)}{2}}\res(f,f').
\end{equation}
Since \eqref{3108:eq011} gives $\lambda n f=-f''+(\lambda x-\mu)f'$, we can write
\[
\res(f,f')
= \frac{\res(\lambda n f,f')}{(\lambda n)^{n-1}}
= \frac{\res(-f''+(\lambda x-\mu)f',f')}{(\lambda n)^{n-1}}
= (-\lambda)^{1-n} n^{3-n} \res(f',f'').
\]
Therefore we get the relation
\begin{equation}\label{3108:eq014}
\res(f,f') = (-\lambda)^{1-n} n^{n} \res\left(\frac{f'}{n},\frac{f''}{n}\right).
\end{equation}
Notice that if we differentiate \eqref{3108:eq011} we obtain
\[
f'''(x) - (\lambda x-\mu)f''(x) + \lambda (n-1) f'(x) = 0.
\]
This means that $f'/n$ solves the same differential equation as $f$,
except that $n$ is replaced by $n-1$.
This allows us to iterate the identity \eqref{3108:eq014} for the resultant,
which leads us to
\[
\res(f,f') = (-1)^{\binom{n}{2}} \lambda^{-\binom{n}{2}} Y(n).
\]
By \eqref{3108:eq013}, we conclude
\[
\Delta = Y(n)\lambda^{-\binom{n}{2}},
\]
or equivalently $\lambda^{\binom{n}{2}}=\Delta^{-1}Y(n)$.

To conclude the proof of Theorem \ref{T2}, consider any
point $(x_1,\dots,x_n)$ on $M$.
The corresponding discriminant $\Delta_0$
will satisfy $\Delta_0\leq \Delta$. If we define $\lambda_0$ by the relation
\begin{equation}\label{2311:eq001}
\lambda_0^{\binom{n}{2}}=\Delta_0^{-1}Y(n),
\end{equation}
we will thus have $\lambda_0\geq \lambda$.
Since $S_1,S_2$ are constant on $M$, the energy $E=nS_2-S_1^2$
is also constant, and we deduce
\[
E = \binom{n}{2}\frac{2n}{\lambda} \geq \binom{n}{2}\frac{2n}{\lambda_0}.
\]
After dividing by the binomial coefficient and raising
to the $\binom{n}{2}$th power, we obtain Theorem~\ref{T2}.

As for the proof of Theorem \ref{T1.7}, we adapt the above argument as follows.
First we note that \eqref{3108:eq012} gives
\[
S_1 = \frac{n\mu}{\lambda},\qquad S_2 = S_1^2 - 2\binom{n}{2}\frac{\mu^2-\lambda}{\lambda^2}.
\]
Therefore
\[
\begin{split}
F(x_1,\dots,x_n)
&=
\frac{a}{n}S_2 + \frac{b}{n^2}S_1^2 + \frac{c}{n}S_1 + d
\\
&=
\frac{a(n-1)}{\lambda} + (a+b)\frac{S_1^2}{n^2} + \frac{cS_1}{n} + d
\geq
\frac{a(n-1)}{\lambda_0} + (a+b)\frac{S_1^2}{n^2} + \frac{cS_1}{n} + d.
\end{split}
\]
Theorem \ref{T1.7} follows by replacing $\lambda_0$ using \eqref{2311:eq001}.


\end{document}